\documentclass[12pt]{article}

\usepackage[utf8]{inputenc}
\usepackage{fullpage}
\usepackage{amsmath,amsthm,amssymb,amsfonts,csquotes,mathtools}
\usepackage{mathrsfs}
\usepackage{graphicx}
\usepackage[colorlinks=true, linkcolor=blue, citecolor=blue, urlcolor=blue, breaklinks=true]{hyperref}
\usepackage{thmtools}
\usepackage[capitalise]{cleveref}
\usepackage{tikz,float}
\usepackage{subcaption}
\usepackage{xcolor}
\usetikzlibrary{decorations.pathreplacing}

\newcommand{\is}{\mathsf{is}}
\newcommand{\ds}{\mathsf{ds}}
\newcommand{\ls}{\mathsf{ls}}
\newcommand{\as}{\mathsf{as}}
\newcommand{\asc}{\operatorname{asc}}
\newcommand{\des}{\operatorname{des}}
\newcommand{\Asc}{\operatorname{Asc}}
\newcommand{\Des}{\operatorname{Des}}
\newcommand{\s}{\mathfrak{S}}
\newcommand{\evac}{\operatorname{evac}}
\newcommand{\dword}{\operatorname{dword}}
\newcommand{\desword}[1]{\mathtt{#1}}

\newtheorem{theorem}{Theorem}[section]

\newtheorem{corollary}[theorem] {Corollary}
\newtheorem{definition}[theorem]{Definition}
\newtheorem{example}[theorem]{Example}
\newtheorem{lemma}[theorem]{Lemma}

\newtheorem{proposition}[theorem]{Proposition}
\newtheorem{remark}[theorem]{Remark}

\theoremstyle{definition}
\newtheorem{algorithm}[theorem]{Algorithm}

\title{\large\textbf{Descent-restricted subsequences via RSK and evacuation}}

\author{Krishna Menon\footnote{Department of Mathematics, KTH Royal Institute of Technology, Sweden. Email: \href{mailto:puzhan@kth.se}{puzhan@kth.se}} \ and Anurag Singh\footnote{Department of Mathematics, Indian Institute of Technology (IIT) Bhilai, India. Email: \href{mailto:anurags@iitbhilai.ac.in}{anurags@iitbhilai.ac.in}}}

\date{}

\begin{document}

\maketitle

\begin{abstract}
    The length $\is(\pi)$ of a longest increasing subsequence in a permutation $\pi$ has been extensively studied. 
    An increasing subsequence is one that has no descents. 
    We study generalizations of this statistic by finding longest subsequences with other descent restrictions. 
    We first consider the statistic which encodes the longest length of a subsequence with a given number of descents. 
    We then generalize this to restrict the descent set of the subsequence. 
    Extending the classical result for $\is(\pi)$, we show how these statistics can be obtained using the RSK correspondence and the Schützenberger involution. 
    In particular, these statistics only depend on the recording tableau of the permutation.
\end{abstract}

\section{Introduction}
The study of longest increasing subsequences (LIS) in permutations is a classical and central topic in combinatorics, with strong connections to representation theory, probability, and theoretical computer science. 
A foundational result of Schensted shows that the length of a longest increasing subsequence of a permutation can be read off from the Robinson--Schensted--Knuth (RSK) correspondence, namely as the length of the first row of the associated Young diagram \cite{schensted}. 
This discovery initiated a vast body of work on variants, distributions, and refinements of LIS (see, for instance, \cite{lisbook,stanley2007increasing}, and the references therein).

A popular generalization is Greene's theorem \cite{greene}, which extends Schensted's result to give interpretations for other row lengths of the Young diagrams obtained using the RSK correspondence. 
This involves the statistic $\is_k(\pi)$, which is the longest length of a subsequence that can be expressed as a union of $k$ increasing subsequences. 
Variants replacing monotonicity with prescribed up-down patterns have also been explored. 
A notable example is longest alternating subsequences, introduced and studied by Stanley \cite{stan_as}. 
There are several other statistics whose motivation comes from LIS \cite{wave,almostinc,common}.

For a permutation $\pi \in \s_n$ and a set $D \subseteq [n - 1]$, we define
\begin{equation*}
        \ls_D(\pi) \coloneqq \max \{\#I \mid \Des(\pi_I) = D\},
\end{equation*}
the length of a longest subsequence of $\pi$ whose descent set is precisely $D$. 
This statistic simultaneously generalizes longest increasing subsequences ($D=\varnothing$), alternating subsequences, and subsequences with a fixed number of descents.

For special cases of these statistics, we describe bounds on them and also characterize and count the permutations that achieve these bounds. 
For the general statistic $\ls_D$, we describe a method (based on a decomposition of $D$ into intervals) to determine when such subsequences exist and how their maximum lengths can be computed. 
This relies on the use of growth diagrams and the Sch\"utzenberger involution \cite{fomin_growth,schutzenberger}, and leads to a complete tableau-theoretic description of these descent-restricted subsequence statistics. 
In particular, for fixed $D$, the value of $\ls_D(\pi)$ depends only on the recording tableau $Q(\pi)$, extending the classical result for longest increasing subsequences.

The paper is organized as follows. 
In \Cref{sec:prelims}, we review the necessary background on the RSK correspondence and Sch\"utzenberger involution. 
\Cref{sec:lsd} studies longest subsequences with a fixed number of descents. 
We first obtain bounds for these statistics and  characterize when they are met. 
We then describe a method to extract these statistics from the evacuation growth diagram. 
In \Cref{sec:lsD}, we extend these results to the statistic $\ls_D$. 
We also characterize when two permutations attain the same value for all these statistics. 
We conclude with several open problems and directions for future work.

\begin{remark}
    The questions posed in \cite[Section 5]{stan_as} were the initial motivation for our study. 
    However, we later found that an algorithm to compute a statistic essentially equivalent to $\ls_D$ (called \emph{longest wave subsequence with trend}) has been defined in \cite{wave}. 
    In particular, \Cref{lsDalgo} described in \Cref{sec:lsD} is similar to the one in \cite{wave}. 
    Since the connections to RSK and the Sch\"utzenberger involution are not explored in \cite{wave}, we present them in this work.
\end{remark}

% \vspace{1cm}
% Longest increasing subsequences of permutations is a popular topic in both mathematics and computer science. 
% One of the significant results in the area is by Schensted \cite{schensted} which states that the length of such a sequence can be obtained using the ubiquitous RSK correspondence. 
% There are several other results as well as open problems related to longest increasing subsequences and we refer the reader to \cite{lisbook} for more information.

% Several variants of longest increasing subsequences have been studied: 

\section{Preliminaries}\label{sec:prelims}

We use $\s_n$ to denote the set of permutations of $[n] \coloneqq \{1, 2, \ldots, n\}$.
Permutations are written in one-line notation. 
For a permutation $\pi \in \s_n$ and a subset $I \subseteq [n]$, we write $\pi_I$
for the subsequence of $\pi$ given by $\pi(i_1)\pi(i_2)\cdots\pi(i_k),$ 
where $i_1 < i_2 < \cdots < i_k$ are the elements of $I$.

We collect below some notation that will be used throughout the paper.
For any permutation $\pi \in \s_n$ (or, more generally, for any sequence of length
$n$), we use the following conventions.
\begin{itemize}
	\item $\Des(\pi) = \{ i \in [n-1] \mid \pi(i) > \pi(i+1) \}$ denotes the descent
	set of $\pi$, and $\des(\pi) = \#\Des(\pi)$ denotes the number of descents.
	We define the ascent set $\Asc(\pi)$ and the number of ascents $\asc(\pi)$
	analogously.
	
	\item $\is(\pi)$ denotes the length of a longest increasing subsequence of $\pi$.
	
	\item For $k \ge 1$, $\is_k(\pi)$ denotes the maximum length of a subsequence of
	$\pi$ that can be written as the union of $k$ increasing subsequences.
\end{itemize}

\begin{example}
	Let $\pi = 234615$. If $I = \{2,4,5\}$, then $\pi_I = 361$.
	Moreover, $\Des(\pi) = \{4\}$ and $\des(\pi) = 1$.
	The length of a longest increasing subsequence of $\pi$ is $\is(\pi) = 4$; for
	instance, the subsequence $2345$ is increasing.
	
	Furthermore, $\is_2(\pi) = 6$, since the entire permutation $\pi$ can be written as
	the union of two increasing subsequences, for example $2346$ and $15$.
\end{example}

\subsection{The RSK correspondence}

We now recall some basic facts about the Robinson--Schensted--Knuth (RSK) correspondence defined in \cite{schensted}. 
The definitions and results we state can be found in \cite{stanley_ec2}.

A \emph{partition} $\lambda$ of $n$, denoted $\lambda \vdash n$, is a weakly decreasing sequence of positive numbers that sums to $n$. 
The \emph{Young diagram} of a partition $\lambda = (\lambda_1, \lambda_2, \ldots)$ consists of rows of boxes where the $i$-th row (from the top) has $\lambda_i$ boxes. 
A \emph{standard Young tableau (SYT)} of shape $\lambda$ is a filling of the boxes in the Young diagram of a partition of $n$ with the numbers from $[n]$ (each used once) such that the rows and columns are increasing.

Given a permutation $\pi \in \s_n$, the RSK correspondence associates to $\pi$ a pair $(P(\pi), Q(\pi))$ of SYTs of the same shape $\lambda \vdash n$. 
The \emph{insertion tableau} $P(\pi)$ is obtained by performing row-insertion on the entries $\pi(1), \pi(2), \ldots, \pi(n)$ in order, while the \emph{recording tableau} $Q(\pi)$ records the positions at which new boxes are added in the insertion process. 
We will not go through the details of this well-known correspondence but only mention the results we will use.

If $(\lambda_1, \lambda_2, \ldots)$ is the shape of $P(\pi)$, which is the same as that of $Q(\pi)$, then $\is_k(\pi) = \lambda_1 + \lambda_2 + \cdots + \lambda_k$. 
If the first row of $Q(\pi)$ is $u_1 < u_2 < \cdots < u_{\is(\pi)}$, then $u_i$ is the leftmost (smallest) index where an increasing subsequence of $\pi$ of length $i$ can end, i.e., $\is(\pi(1)\pi(2) \cdots \pi(u_i)) = i$ and $\is(\pi(1)\pi(2) \cdots \pi(k)) < i$ for all $k < u_i$. 
Note that while performing the RSK algorithm, $u_i$ is the index at which the length of the first row becomes $i$.

For an SYT $Q$, we define $\Des(Q)$ to be the set of entries $i$ in $Q$ such that $i + 1$ is in a strictly lower row than $i$ and set $\des(Q) = \#\Des(Q)$. 
For any permutation $\pi$, we have $\Des(\pi) = \Des(Q(\pi))$.

\begin{example}\label{rskex}
    For the permutation $\pi = 4 3 6 5 1 7 2$, we have the following.
    \begin{center}
        \begin{tikzpicture}[scale = 0.8]
            \draw (0, 0) grid (2, -3);
            \draw (2, -1) grid (3, 0);
            \node at (-1.15, -1) {$P(\pi) =$};
            \begin{scope}[shift = {(0.5, -0.5)}]
                \node at (0, 0) {$1$};
                \node at (1, 0) {$2$};
                \node at (2, 0) {$7$};
                \node at (0, -1) {$3$};
                \node at (1, -1) {$5$};
                \node at (0, -2) {$4$};
                \node at (1, -2) {$6$};
            \end{scope}
            \begin{scope}[shift = {(6.5, 0)}]
                \draw (0, 0) grid (2, -3);
                \draw (2, -1) grid (3, 0);
                \node at (-1.15, -1) {$Q(\pi) =$};
                \begin{scope}[shift = {(0.5, -0.5)}]
                    \node at (0, 0) {$1$};
                    \node at (1, 0) {$3$};
                    \node at (2, 0) {$6$};
                    \node at (0, -1) {$2$};
                    \node at (1, -1) {$4$};
                    \node at (0, -2) {$5$};
                    \node at (1, -2) {$7$};
                \end{scope}
            \end{scope}
        \end{tikzpicture}
    \end{center}
    This gives $\is(\pi) = 3$, $\is_2(\pi) = 5$, and $\is_k(\pi) = 7$ for all $k \geq 3$. 
    The second entry of the first row the recording tableau $Q(\pi)(1, 2) = 3$. 
    This reflects the fact that $\is(43) = 1$ and $\is(436) = 2$. 
    We also have $\Des(Q(\pi)) = \{1, 3, 4, 6\} = \Des(\pi)$.
\end{example}

\subsection{The Schützenberger involution}

An important operation on SYTs that we will use is the Schützenberger involution \cite{schutzenberger}. 
This in a shape-preserving involution on SYTs and the image of an SYT $Q$ under this map is called its \emph{evacuation tableau} and denoted by $\evac(Q)$. 
We use Fomin's \emph{growth diagrams} to define this involution \cite{fomin_growth} (other equivalent definitions can be found in \cite{stanley_ec2}).

For two partitions $\lambda$ and $\mu$, we say that $\lambda$ contains $\mu$ if the Young diagram of $\lambda$ contains that of $\mu$. 
To start, we note that any SYT can be encoded using a chain of partitions. 
These partitions indicate the shapes formed by boxes labeled with numbers having value at most $i$ for each $i \geq 0$. 
For example, the SYT $P(\pi)$ from \Cref{rskex} is encoded by the chain
\begin{equation*}
    \varnothing \subset 1 \subset 2 \subset 21 \subset 211 \subset 221 \subset 222 \subset 322.
\end{equation*}

The evacuation growth diagram associated to an SYT $Q$ of size $n$ is an triangle of partitions $\Lambda_{i, j}$ for $0 \leq i \leq j \leq n$. 
The tableau $Q$ is encoded using $(\Lambda_{0, j})_{j = 0}^n$ and we set $\Lambda_{i, 0} = \varnothing$ for all $i$. 
We want the final diagram to be such that $\Lambda_{i, j}$ is contained in $\Lambda_{i, j + 1}$ and contains $\Lambda_{i + 1, j}$. 
Note that if $\lambda \vdash n$ and $\mu \vdash n + 2$, then there are at most two partitions that contain $\lambda$ and are contained in $\mu$. 
If $\Lambda_{i - 1, j - 1}$, $\Lambda_{i, j - 1}$, and $\Lambda_{i - 1, j}$ are defined, we set $\Lambda_{i, j}$ to be different from $\Lambda_{i - 1, j - 1}$ whenever possible. 
These conditions uniquely define the triangle of partitions. 
We define $\evac(Q)$ to be the SYT that is encoded by the chain of partitions $(\Lambda_{n - i, n})_{i = 0}^n$.

\begin{example}\label{evacex}
    An SYT and its evacuation tableau are shown below.
    \begin{center}
        \begin{tikzpicture}[scale = 0.8]
            \draw (0, 0) grid (3, -1);
            \draw (0, -1) grid (1, -3);
            \node at (-1, -1) {$Q =$};
            \begin{scope}[shift = {(0.5, -0.5)}]
                \node at (0, 0) {$1$};
                \node at (1, 0) {$2$};
                \node at (2, 0) {$4$};
                \node at (0, -1) {$3$};
                \node at (0, -2) {$5$};
            \end{scope}
            \begin{scope}[shift = {(7, 0)}]
                \draw (0, 0) grid (3, -1);
                \draw (0, -1) grid (1, -3);
                \node at (-1.4, -1) {$\evac(Q) =$};
                \begin{scope}[shift = {(0.5, -0.5)}]
                    \node at (0, 0) {$1$};
                    \node at (1, 0) {$3$};
                    \node at (2, 0) {$5$};
                    \node at (0, -1) {$2$};
                    \node at (0, -2) {$4$};
                \end{scope}
            \end{scope}
        \end{tikzpicture}
    \end{center}
    The evacuation growth diagram for $Q$ is shown in \Cref{fig:evac_example}. 
    The diagram is drawn so that if it is rotated $45$ degrees clockwise, then the labeling of the partitions $\Lambda_{i, j}$ match those of a matrix. 
    In this case, we have $\Lambda_{1, 5} = 211$ and $\Lambda_{2, 4} = 2$. 
    The row on the left encodes $Q$ and the column on the right encodes $\evac(Q)$.
\end{example}

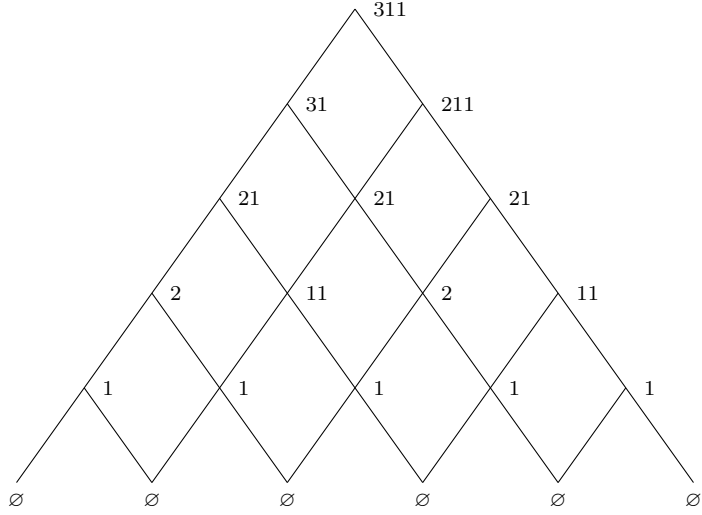
\begin{figure}[ht]
	\centering
	\begin{tikzpicture}[scale = 0.9, yscale = 1.4]
		
		\foreach \x in {0,...,5}
		\node [below] at (2*\x, 0) {\scriptsize $\varnothing$};
		
		\foreach \x in {0,...,4}
		\node [right = 0.1cm] at (2*\x + 1, 1) {\scriptsize $1$};
		
		\node [right = 0.1cm] at (2,2) {\scriptsize $2$};
		\node [right = 0.1cm] at (4,2) {\scriptsize $11$};
		\node [right = 0.1cm] at (6,2) {\scriptsize $2$};
		\node [right = 0.1cm] at (8,2) {\scriptsize $11$};
		
		\node [right = 0.1cm] at (3,3) {\scriptsize $21$};
		\node [right = 0.1cm] at (5,3) {\scriptsize $21$};
		\node [right = 0.1cm] at (7,3) {\scriptsize $21$};
		
		\node [right = 0.1cm] at (4,4) {\scriptsize $31$};
		\node [right = 0.1cm] at (6,4) {\scriptsize $211$};
        
		\node [right = 0.1cm] at (5,5) {\scriptsize $311$};
		\foreach \x in {0,...,4}
		\draw (2*\x, 0) -- (5 + \x, 5 - \x);
		\foreach \x in {0,...,4}
		\draw (2*5 - 2*\x, 0) -- (5 - \x, 5 - \x);		
	\end{tikzpicture}
	\caption{Evacuation growth diagram for $Q$ defined in \Cref{evacex}.}
	\label{fig:evac_example}
\end{figure}

We now gather the facts about this involution on SYTs that we make use of in the sequel. The proof of these facts can be found in \cite{stanley_ec2}. 
For any permutation $\pi \in \s_n$, we have $P(\pi^{rc}) = \evac(P(\pi))$ and $Q(\pi^{rc}) = \evac(Q(\pi))$. 
Here, $\pi^{rc}$ is the \emph{reverse-complement} of $\pi$. 
It is obtained by first reversing $\pi$, then replacing $i$ with $n + 1 - i$ for all $i \in [n]$. 
As a consequence, note that if the first row of $\evac(Q(\pi))$ is $v_1 < v_2 < \cdots < v_{\is(\pi)}$, then $n - v_i + 1$ is the rightmost (largest) index where an increasing subsequence of $\pi$ of length $i$ can start. 
The evacuation growth diagram of $Q(\pi)$ also encodes the shapes obtained when RSK is applied to factors of $\pi$. 
Precisely, $\Lambda_{i - 1, j}$ is the shape of $Q(\pi_{[i, j]})$. 
These facts can be verified for $\pi = 3 4 2 5 1$ whose recording tableau $Q(\pi)$ is the same as the SYT $Q$ from \Cref{evacex}.

\section{Restricting the number of descents}\label{sec:lsd}

We first consider subsequences of a permutation with a prescribed number of descents.

\begin{definition}
	For any permutation $\pi$ and $d \geq 0$,
	\begin{equation*}
		\ls_d(\pi) \coloneqq \max \{\#I \mid \des(\pi_I) = d\}.
	\end{equation*}
\end{definition}

Hence, $\ls_d(\pi)$ is the maximum length of a subsequence of $\pi$ that has exactly $d$ descents. 
When $d = 0$, this reduces to the longest increasing subsequence statistic, so $\ls_0(\pi) = \is(\pi)$.

\begin{example}
	Let $\pi = 1 5 7 3 4 2 6$. Then the values of $\ls_d(\pi)$ are as follows.
	\begin{itemize}
		\item $\ls_0(\pi) = \is(\pi) = 4$. 
		Consider the subsequence $1346$.
		
		\item $\ls_1(\pi) = 6$. Consider $157346$.
		
		\item $\ls_2(\pi) = 7$, since the entire permutation $\pi$ has exactly two descents.
		
		\item $\ls_d(\pi) = 0$ for all $d > 2$, because no subsequence of $\pi$ can have more descents than $\pi$ itself.
	\end{itemize}
\end{example}

The following lemma records a few basic properties of these statistics which will be used repeatedly in the sequel.
% and provide useful intuition about the behavior of $\ls_d(\pi)$.

\begin{lemma}\label{lsdnonzero}
	For any permutation $\pi$, we have $\ls_d(\pi) \geq d + 1$ for all
	$d \leq \des(\pi)$ and $\ls_d(\pi) = 0$ for all $d > \des(\pi)$. Moreover, for any $d_1 < d_2 \leq \des(\pi)$, we have $\ls_{d_1}(\pi) < \ls_{d_2}(\pi)$.
\end{lemma}

\begin{proof}
	Let $\Des(\pi) = \{i_1 < i_2 < \cdots < i_{\des(\pi)}\}$.
	For any $d \leq \des(\pi)$, consider the subsequence $\pi(1)\pi(2)\cdots\pi(i_d)\pi(i_d+1)$. 
	This subsequence has exactly $d$ descents, since each descent of $\pi$ occurring at a position in $\{i_1,\ldots,i_d\}$ is preserved.
	As $i_d \geq d$, this subsequence has length at least $d+1$, and hence
	$\ls_d(\pi) \geq d+1$.
	
	Any subsequence of $\pi$ can have at most $\des(\pi)$ descents, so
	$\ls_d(\pi) = 0$ for all $d > \des(\pi)$. 
	Finally, the strict inequality $\ls_{d_1}(\pi) < \ls_{d_2}(\pi)$ for
	$d_1 < d_2 \leq \des(\pi)$ follows from the fact that inserting one number into a sequence either does not change the number of descents or increases it by $1$.
\end{proof}

We next establish general bounds on the statistics $\ls_d(\pi)$ in terms of classical subsequence parameters associated with $\pi$.

\begin{lemma}\label{lsdbound}
    For any permutation $\pi$ and $d \leq \des(\pi)$, we have
    \begin{equation*}
        \is(\pi) + d \leq \ls_d(\pi) \leq \operatorname{min}\{\is_{d + 1}(\pi), \asc(\pi) + d + 1\}.
    \end{equation*}
\end{lemma}

\begin{proof}
    Any subsequence of $\pi$ with $d$ descents is the union of $d + 1$ increasing subsequences and hence $\ls_d(\pi) \leq \is_{d + 1}(\pi)$. 
    To obtain the other upper bound, break up the permutation $\pi$ into maximal decreasing factors. 
    If $\pi = 741632985$, it breaks up into maximal decreasing factors as $741\,|\,632\,|\,985$. 
    Note that the number of such factors is $\asc(\pi) + 1$. 
    By the pigeon-hole principle, any subsequence with more than $\asc(\pi) + d + 1$ terms has at least $d + 1$ descents. 
    This shows that $\ls_d(\pi) \leq \asc(\pi) + d + 1$.
    
    The lower bound follows from \Cref{lsdnonzero}. 
    Since $d \leq \des(\pi)$, we have $\is(\pi) = \ls_0(\pi) < \ls_1(\pi) < \cdots < \ls_d(\pi)$ which gives us $\is(\pi) + d \leq \ls_d(\pi)$.
    % Starting with a longest increasing subsequence $\sigma$ of $\pi$, and including $d$ terms of $\pi$ that are not in $\sigma$, results in a subsequence with at most $d$ descents (since $d \leq \des(\pi)$, such terms can be found). 
    % Hence, $\ls_d(\pi) \geq \is(\pi) + d$.
\end{proof}

The bounds in \Cref{lsdbound} are not sharp in general. We characterize precisely when the lower bound $\ls_d(\pi) = \is(\pi) + d$ is attained.

\begin{proposition}\label{gs=isplus1}
    For any permutation $\pi$ and $1 \leq d \leq \des(\pi)$, we have $\ls_d(\pi) = \is(\pi) + d$ if and only if $\asc(\pi) = \is(\pi) - 1$.
\end{proposition}

\begin{proof}
    If $\asc(\pi) = \is(\pi) - 1$, then \Cref{lsdbound} immediately gives $\ls_d(\pi) = \is(\pi) + d$. 
    Let $\pi \in \s_n$. 
    Note that if there exists $1 \leq i < j \leq n$ such that $\pi(i) < \pi(j)$, then there exists an ascent at some index in $[i, j - 1]$. 
    Using an increasing subsequence of length $\is(\pi)$, this shows that $\asc(\pi) \geq \is(\pi) - 1$.
    
    Hence, by \Cref{lsdbound}, we have to show that if $\asc(\pi) \geq \is(\pi)$, then $\ls_d(\pi) \geq \is(\pi) + d + 1$. 
    We only have to prove this for the case $d = 1$. 
    Since if $\ls_1(\pi) \geq \is(\pi) + 2$, from \Cref{lsdnonzero}, we get $\ls_1(\pi) < \ls_2(\pi) < \cdots < \ls_d(\pi)$, which gives us $\ls_d(\pi) \geq \is(\pi) + d + 1$, as required.

    Set $k = \is(\pi)$ and suppose that $\asc(\pi) \geq k$. 
    We will show that $\ls_1(\pi) \geq k + 2$ by constructing a subsequence of length $k + 2$ with exactly one descent. 
    Suppose that $I = \{i_1 < i_2 < \cdots < i_k\}$ is such that $\pi_I$ is an increasing subsequence of $\pi$.

    If $\pi$ has an ascent $j \in [1, i_1 - 1] \cup [i_k, n]$, then $\pi_{I \cup \{j, j + 1\}}$ is a subsequence of length $k + 2$ with one descent. 
    Note that such a $j$ can never be equal to $i_1 - 1$ or $i_k$ since this would contradict the fact that $\is(\pi) = k$.

    If $\pi$ has no ascent of the form mentioned above, then since $\asc(\pi) \geq k$, there must exist some $m \in [k - 1]$ such that $\pi$ has at least two ascents in $[i_m, i_{m + 1} - 1]$. 
    Note that $\is(\pi) = k$ implies that for any $i \in [i_m, i_{m + 1} - 1]$, $\pi(i) < \pi(i_m)$ or $\pi(i) > \pi(i_{m + 1})$. 
    The fact that there are at least two ascents in $[i_m, i_{m + 1} - 1]$ means that there exist $i, j \in [i_m, i_{m + 1} - 1]$ where $i < j$ such that
    \begin{enumerate}
        \item $\pi(i) > \pi(i_{m + 1})$ and $\pi(j) < \pi(i_m)$, or
        \item $\pi(i_{m + 1}) < \pi(i)< \pi(j)$, or
        \item $\pi(i) < \pi(j) < \pi(i_m)$.
    \end{enumerate}
    In each case, $\pi_{I \cup \{i, j\}}$ is a subsequence of length $k + 2$ with one descent.
\end{proof}

The above proposition, along with \Cref{lsdbound}, gives us the following result.

\begin{corollary}
    For any permutation $\pi$ and $d \leq \des(\pi)$, if $\asc(\pi) = \is(\pi)$, then $\ls_d(\pi) = \is(\pi) + d + 1$.
\end{corollary}

However, the converse of the above corollary is not true. 
For example, for $\pi = 5 6 3 4 1 2$, we have $\is(\pi) = 2$, $\ls_1(\pi) = 4$, and $\asc(\pi) = 3$.

We now count the permutations described in \Cref{gs=isplus1}. 
To do this, we will need to recall the definition of a \emph{semistandard Young tableau (SSYT)}. 
An SSYT of shape $\lambda$ is a filling of the Young diagram of $\lambda$ with positive integers such that the rows are weakly increasing and columns are strictly increasing. 
For any partition $\lambda$, we use $f^{\lambda}$ to denote the number of SYTs of shape $\lambda$; and for any $k \geq 1$, we use $s_\lambda(1^k)$ to denote the number of SSYTs of shape $\lambda$ with entries from $[k]$. 
This notation for the number of SSYTs is a reflection of the fact that these numbers can be extract using Schur functions (see \cite[Chapter 7]{stanley_ec2}).

\begin{proposition}
    For any $n \geq 1$, we have
    \begin{equation*}
        \#\{\pi \in \s_n \mid \asc(\pi) = \is(\pi) - 1\} = \sum_{\lambda \vdash n} f^{\lambda}s_{\lambda'}(1^{\lambda_1}).
    \end{equation*}
    This sequence of numbers is given by \cite[\href{https://oeis.org/A268699}{A268699}]{oeis}.
\end{proposition}

\begin{proof}
    We prove the result by showing that for any $\lambda \vdash n$, the number of SYTs $Q$ of shape $\lambda$ with $\des(Q) + \lambda_1 = n$ is given by $s_{\lambda'}(1^{\lambda_1})$. 
    The result follows by the RSK correspondence.

    Fix a partition $\lambda \vdash n$. 
    By transposing the SSYTs, we can think of $s_{\lambda'}(1^{\lambda_1})$ as counting tableaux of shape $\lambda$ such that
    \begin{itemize}
        \item entries strictly increase along rows,
        \item entries weakly increase along columns, and
        \item all entries are from $[\lambda_1]$.
    \end{itemize}
    We will exhibit a bijection between such tableaux and the SYTs we want to count.
    
    Starting with an SYT $Q$ of shape $\lambda$ such that $\des(Q) + \lambda_1 = n$ we construct such a tableau as follows: 
    Suppose the entries of the first row of $Q$ are $u_1 < u_2 < \cdots < u_{\lambda_1}$. 
    For each $i \in [\lambda_1 - 1]$, replace all entries of $Q$ in $[u_i, u_{i + 1} - 1]$ with $i$. 
    Also replace all entries in $[u_{\lambda_1}, n]$ with $\lambda_1$. 
    An instance of this mapping is
    \vspace{0.2cm}
    \begin{center}
        \begin{tikzpicture}[scale = 0.8]
            \draw (0, 0) grid (2, -3);
            \draw (2, -1) grid (3, 0);
            \begin{scope}[shift = {(0.5, -0.5)}]
                \node at (0, 0) {$1$};
                \node at (1, 0) {$3$};
                \node at (2, 0) {$6$};
                \node at (0, -1) {$2$};
                \node at (1, -1) {$4$};
                \node at (0, -2) {$5$};
                \node at (1, -2) {$7$};
            \end{scope}
            \node at (4, -1.5) {$\rightarrow$};
            \begin{scope}[shift = {(5, 0)}]
                \draw (0, 0) grid (2, -3);
                \draw (2, -1) grid (3, 0);
                \begin{scope}[shift = {(0.5, -0.5)}]
                    \node at (0, 0) {$1$};
                    \node at (1, 0) {$2$};
                    \node at (2, 0) {$3$};
                    \node at (0, -1) {$1$};
                    \node at (1, -1) {$2$};
                    \node at (0, -2) {$2$};
                    \node at (1, -2) {$3$};
                \end{scope}
                \node at (4, -1.5) {.};
            \end{scope}
        \end{tikzpicture}
    \end{center}
    
    Note that $\des(Q) = \#\{i \in [2, n] \mid i - 1\text{ is in a higher row in $Q$ than }i\}$ and $u_1, \ldots, u_{\lambda_1}$ are not in this set. 
    Hence, since $\des(Q) + \lambda_1 = n$, all other entries in $Q$ must be in this set. 
    This fact can be used to show that the map described above is a bijection. 
    The relation to \cite[\href{https://oeis.org/A268699}{A268699}]{oeis} follows using properties of the RSK algorithm applied to words (instead of permutations).
\end{proof}

We now describe a method to compute the statistic $\ls_d$ on permutations. 
For a permutation $\pi$ of size $n$ and $d \leq \des(\pi)$, we have
\begin{equation*}
    \ls_d(\pi) = \operatorname{max}\{\is(\pi_{[1, i_1]}) + \is(\pi_{[i_1 + 1, i_2]}) + \cdots + \is(\pi_{[i_d + 1, n]}) \mid 1 \leq i_1 < i_2 < \cdots < i_d \leq n\}.
\end{equation*}
Note that for any $1 \leq i \leq j \leq n$, the growth diagram for $\evac(Q(\pi_{[i, j]}))$ is a sub-diagram of the one for $\evac(Q(\pi))$ (see \Cref{dyckongrowth}). 
This shows that we can obtain $\ls_d(\pi)$ from the growth diagram of $\evac(Q(\pi))$. 
We vary over all `break-ups' of the diagram into $d + 1$ sub-diagrams as shown in \Cref{dyckongrowth}. 
For each such break-up, we sum the first parts of the partitions at the peaks of the red `hills'. 
The maximum over all these values is $\ls_d(\pi)$. 
Hence $\ls_d(\pi)$ only depends on $Q(\pi)$.

This is a special case of the more general statistic $\ls_D(\pi)$ (see \Cref{lsDalgo}) that can be read off from the growth diagram.

\begin{figure}[ht]
    \centering
    \begin{tikzpicture}[scale = 0.9, yscale = 1.4]
        \foreach \x in {0,...,7}
        \node [below] at (2*\x, 0) {\scriptsize $\varnothing$};
        \foreach \x in {0,...,6}
        \node [right = 0.1cm] at (2*\x + 1, 1) {\scriptsize $1$};
    
        \node [right = 0.1cm] at (2, 2) {\scriptsize $2$};
        \node [right = 0.1cm] at (4, 2) {\scriptsize $11$};
        \node [right = 0.1cm] at (6, 2) {\scriptsize $2$};
        \node [right = 0.1cm] at (8, 2) {\scriptsize $11$};
        \node [right = 0.1cm] at (10, 2) {\scriptsize $2$};
        \node [right = 0.1cm] at (12, 2) {\scriptsize $2$};
    
        \node [right = 0.1cm] at (3, 3) {\scriptsize $21$};
        \node [right = 0.1cm] at (5, 3) {\scriptsize $21$};
        \node [right = 0.1cm] at (7, 3) {\scriptsize $21$};
        \node [right = 0.1cm] at (9, 3) {\scriptsize $21$};
        \node [right = 0.1cm] at (11, 3) {\scriptsize $3$};
    
        \node [right = 0.1cm] at (4, 4) {\scriptsize $31$};
        \node [right = 0.1cm] at (6, 4) {\scriptsize $211$};
        \node [right = 0.1cm] at (8, 4) {\scriptsize $22$};
        \node [right = 0.1cm] at (10, 4) {\scriptsize $31$};
    
        \node [right = 0.1cm] at (5, 5) {\scriptsize $311$};
        \node [right = 0.1cm] at (7, 5) {\scriptsize $221$};
        \node [right = 0.1cm] at (9, 5) {\scriptsize $32$};
    
        \node [right = 0.1cm] at (6, 6) {\scriptsize $321$};
        \node [right = 0.1cm] at (8, 6) {\scriptsize $321$};
    
        \node [right = 0.1cm] at (7, 7) {\scriptsize $331$};
    
        \foreach \x in {0,...,6}
        \draw (2*\x, 0) -- (7 + \x, 7 - \x);
        \foreach \x in {0,...,6}
        \draw (2*7 - 2*\x, 0) -- (7 - \x, 7 - \x);

        \draw [ultra thick, red] (0, 0) -- (2, 2) -- (4, 0) -- (8, 4) -- (12, 0) -- (13, 1) -- (14, 0); 
    \end{tikzpicture}
    \caption{If the diagram is the evacuation growth diagram of $\pi$, then the three red triangles are the evacuation growth diagrams of the factors $\pi_{[1, 2]}$, $\pi_{[3, 6]}$, and $\pi(7)$.}
    \label{dyckongrowth}
\end{figure}

For the case $d = 1$, the observation above states that $\ls_1(\pi)$ can be obtained from the first row of $Q(\pi)$ and that of $\evac(Q(\pi))$.

\begin{proposition}\label{gsfromevac}
    For a non-identity permutation $\pi \in \s_n$, setting $Q = Q(\pi)$, we have
    \begin{equation*}
        \ls_1(\pi) = \operatorname{max}\{ i + j \mid Q(1, i) + \evac(Q)(1, j) \leq n\}.
    \end{equation*}
\end{proposition}

\begin{proof}
    Suppose $\pi$ has a subsequence of length $i + j$, which has exactly one descent, which is at $i$. 
    Note that the first $i$ terms of this subsequence forms an increasing subsequence of length $i$ and similarly for the last $j$ terms. 
    Using the interpretation of the first row of the recording tableau of a permutation, we get $Q(1, i) < n - \evac(Q)(1, j) + 1$.

    Similarly, if $Q(1, i) < n - \evac(Q)(1, j) + 1$, then we can construct a subsequence of size $i + j$ by combining an increasing subsequence of length $i$ that ends at $Q(1, i)$ and one of length $j$ starting at $n - \evac(Q)(1, j) + 1$. 
    Since $\des(\pi) \geq 1$ and this subsequence has at most one descent, \Cref{lsdnonzero} gives $\ls_1(\pi) \geq i + j$.
\end{proof}

\begin{example}
    With $\pi = 5 3 1 6 2 7 4$, the first row of $Q(\pi)$ and $\evac(Q(\pi))$ are $1, 4, 6$ and $1, 3, 5$ respectively. 
    Hence, we have $\ls_1(\pi) = \max\{1 + 1, 1+ 2, 1 + 3, 2 + 1, 2 + 2, 3 + 1\} = 4$.
\end{example}

\section{Restricting the descent set}\label{sec:lsD}

We now consider the finer statistic where we restrict the descent set rather than just the number of descents.

\begin{definition}
    For a permutation $\pi$ and set of numbers $D$,
    \begin{equation*}
        \ls_D(\pi) \coloneqq \max \{\#I \mid \Des(\pi_I) = D\}.
    \end{equation*}
\end{definition}

Hence, $\ls_D(\pi)$ is the length of a longest subsequence of $\pi$ whose descent set is $D$.

\begin{example}
    For $\pi = 4 2 7 8 3 5 6 1$, we have the following.
    \begin{itemize}
        \item $\ls_{\{3\}}(\pi) = 6$. Consider the subsequence $4 7 8 3 5 6$.

        % \item $\ls_{\{2, 4\}}(\pi) = 5$. 
        % Consider the subsequence $2 8 3 6 1$.

        \item $\ls_{\{1, 2, 3\}}(\pi) = 0$ since there is no subsequence with descent set $\{1, 2, 3\}$.

        \item $\ls_{\{1, 3, 4\}}(\pi) = 5$. 
        Consider $4 2 8 3 1$.
    \end{itemize}
\end{example}

Note that the statistics on permutations we have seen so far can be obtained from this one. 
For example, $\is(\pi) = \ls_{\varnothing}(\pi)$ and $\ls_d(\pi) = \operatorname{max}\{\ls_D(\pi) \mid \#D = d\}$.

The statistics defined in \cite{stan_as} can also be obtained from the statistics mentioned above. 
For a permutations $\pi$, set $\as(\pi)$ to be the length of a longest \emph{alternating} subsequence of $\pi$. 
We say that a sequence $\sigma$ of length $n$ is alternating if $\Des(\sigma) = \{1, 3, 5, \ldots\} \cap [n - 1]$. 
Hence, $\as(\pi) = \max \{\#I \mid \pi_I\text{ is alternating}\}.$

To obtain $\as(\pi)$, find the largest $k$ such that $\ls_D(\pi) \neq 0$ where $D = \{1, 3, 5, \ldots, 2k - 1\}$. 
If $\ls_D(\pi) = 2k$, then $\as(\pi) = 2k$. 
Otherwise, $\as(\pi) = 2k + 1$.

Another class of statistics defined in \cite[Section 5]{stan_as} is via \emph{descent words}.

\begin{definition}\label{desword}
    For any sequence $\pi$ of length $n$, its descent word $\dword(\pi)$ is given by $w_1w_2 \cdots w_{n - 1}$ where $w_i = \desword{D}$ if $i \in \Des(\pi)$ and otherwise $w_i = \desword{U}$.
\end{definition}

For example, $\dword(534612) = \desword{DUUDU}$. 
Starting with any finite word $w$ in the alphabet $\{\desword{U}, \desword{D}\}$, set $\operatorname{len}_w(\pi)$ to be the longest length of a subsequence of $\pi$ whose descent word is a prefix of $w^\infty =www\cdots$, the infinite word formed by concatenating copies of $w$. 
Note that $\is(\pi) = \operatorname{len}_{\desword{U}}(\pi)$ and $\as(\pi) = \operatorname{len}_{\desword{DU}}(\pi)$.

\begin{example}
    For $\pi = 31452867$, we have $\operatorname{len}_{\desword{UUD}}(\pi) = 6$. 
    Note that $(\desword{UUD})^\infty = \desword{UUDUUDUUD}\cdots$ and $\pi$ has the subsequence $\sigma = 345267$ with $\dword(\sigma) = \desword{UUDUU}$. 
    Also, $\pi$ has no subsequence with descent word $\desword{UUDUUD}$.
\end{example}

Determining whether $\pi$ has a subsequence with a given descent word can be done using statistics of the form $\ls_D$. 
For example, $\pi$ has a subsequence with descent word $\desword{DUDUU}$ if and only if $\ls_{\{1, 3\}}(\pi) \geq 6$. 
This can be used to show that for any word $w$ in the alphabet $\{\desword{U}, \desword{D}\}$, we can obtain $\operatorname{len}_w$ from statistics of the form $\ls_D$.

\subsection{Exactly one descent}

Before moving to the general case, we first consider the simplest case of when $\#D = 1$. 
We first determine when these statistics attain a non-zero value.

\begin{proposition}\label{gsinonzero}
    Let $\pi$ be a non-identity permutation of size $n$ with last descent at $k$. 
    If $j = \is(\pi_{[k]})$, then $\ls_{\{i\}}(\pi) \geq i + n - k$ for all $i \leq j$ and $\ls_{\{i\}}(\pi) = 0$ for all $i > j$.
\end{proposition}

\begin{proof}
    Let $i \leq j$. 
    Choose an increasing subsequence of $\pi(1)\pi(2) \cdots \pi(k)$ that is of length $i$ (which is possible by the definition of $j$). 
    If the last term of this subsequence is $\pi(m)$ and $\pi(m) > \pi(k)$, then appending $\pi(k + 1)\pi(k + 2) \cdots \pi(n)$ to the end of this subsequence gives a subsequence of length $i + n - k$ with descent set $\{i\}$. 
    If $\pi(m) \leq \pi(k)$, then deleting $\pi(m)$ and appending $\pi(k)\pi(k + 1) \cdots \pi(n)$ gives us a subsequence of the required form. 
    Hence, we have $\ls_{\{i\}}(\pi) \geq i + n - k$.

    If $i > j$, then any increasing subsequence of length $i$ (if one exists) must have last entry $\pi(m)$ for some $m > k$. 
    But no term after $\pi(m)$ is smaller than it (by the definition of $k$). 
    Hence, there cannot exist a subsequence with descent set $\{i\}$.
\end{proof}

Just as we were able to extract $\ls_1(\pi)$ from the first row of $Q(\pi)$ and its evacuation tableau (see \Cref{gsfromevac}), we can also do so for $\ls_{\{i\}}(\pi)$.

\begin{proposition}\label{gsievac}
    Let $\pi$ be a permutation of size $n$ and set $Q = Q(\pi)$ and $Q_e = \evac(Q)$. 
    For any $i \geq 1$, we have
    \begin{equation*}
        \ls_{\{i\}}(\pi) = \operatorname{max}\{i + j \mid Q(1, i) + Q_e(1, j) \leq n\text{ and }Q_e(1, j + 1) \neq Q_e(1, j) + 1\}
    \end{equation*}
    where we set the maximum of the empty set to be $0$.
\end{proposition}

\begin{proof}
    For any $i, j \geq 1$, we use $u_i$ to denote $Q(1, i)$ and $v_j$ to denote $n - Q_e(1, j) + 1$. 
    We say that a pair $(i, j)$ is a \emph{good pair} if $u_i < v_j$ and $v_{j + 1} \neq v_j - 1$.
    Note that for any $j \geq 1$, $v_{j + 1} \neq v_j - 1$ is equivalent to $v_j - 1 \in \Des(\pi)$. 
    Hence, the proposition asserts that for any $i \geq 1$, $$\ls_{\{i\}}(\pi) = \operatorname{max}\{i + j \mid (i, j)\text{ is a good pair}\}.$$
    
    Let $(i, j)$ be a good pair. 
    We will construct a subsequence with descent set $\{i\}$ of length $i + j$. 
    Consider an increasing subsequence of length $i$ ending at $u_i$ and combine it with an increasing subsequence of length $j$ starting at $v_j$. 
    If $\pi(u_i) > \pi(v_j)$, we get a subsequence satisfying our requirements. 
    If $\pi(u_i) < \pi(v_j)$, replace $\pi(u_i)$ in this subsequence with $\pi(v_j - 1)$. 
    This will give us a subsequence that satisfies our requirements.

    Next suppose that we have a subsequence of size $i + k$ with descent set $\{i\}$. 
    We will show that there exists a good pair $(i, j)$ where $j \geq k$. 
    This will prove the proposition.
    
    Suppose that there is no such good pair. 
    Since we already have $u_i < v_k$, this means that $u_i = v_j$ for some $j > k$ and $v_j, v_{j - 1}, \ldots, v_{k + 1}, v_k$ forms a sequence of consecutive numbers. 
    Hence, we have that
    \begin{itemize}
        \item $\pi(v_j) \pi(v_{j - 1}) \cdots \pi(v_{k + 1}) \pi(v_k)$ is an increasing factor of $\pi$,
        \item any increasing subsequence of length $i$ in $\pi$ must end at an index weakly to the right of $u_i = v_j$, and
        \item any increasing subsequence of length $k$ in $\pi$ must begin at an index weakly to the left of $v_k$.
    \end{itemize}
    This tells us that it is impossible for $\pi$ to have a subsequence of length $i + k$ with descent set $\{i\}$, which is a contradiction.
\end{proof}

\begin{example}
    With $\pi = 1 4 5 6 7 2 8 3$, the first row of $Q = Q(\pi)$ is $1, 2, 3, 4, 5, 7$ and the first row of $Q_e = \evac(Q(\pi))$ is $1, 3, 5, 6, 7, 8$. 
    Using the previous proposition, we get $\ls_{\{3\}}(\pi) = 5$. 
    This is because the maximum $j$ such that $Q(1, 3) + Q_e(1, j) \leq 8$ and $Q_e(1, j) \neq Q_e(1, j + 1)$ is $j = 2$.
\end{example}

Using the interpretation of the entries of the first rows of $Q(\pi)$ and $\evac(Q(\pi))$, we get the following equivalent formulation of \Cref{gsievac}.

\begin{proposition}
    Let $\pi$ be a permutation of size $n$ and $i \geq 1$ be such that $\ls_{\{i\}}(\pi) \neq 0$. 
    Suppose $k$ is the smallest number such that $k \in \Des(\pi)$ and $\is(\pi_{[k]}) \geq i$. 
    Then
    \begin{equation*}
        \ls_{\{i\}}(\pi) = i + \is(\pi_{[k + 1, n]}).
    \end{equation*}
\end{proposition}

From definitions, it is clear that for any permutation $\pi$, we have $\ls_{\{i\}}(\pi) \leq \is(\pi) + i$. 
In the case $i = 1$, we characterize when this bound is (not) met.

\begin{corollary}
    For a permutation $\pi \in \s_n$, we have $\ls_{\{1\}}(\pi) < \is(\pi) + 1$ if and only if the last entry of the first row of $\evac(Q(\pi))$ is $n$. 
    The number of such permutations is given by
    \begin{equation*}
        \sum_{\lambda \vdash n - 1} f^{\lambda + 1} f^{\lambda}
    \end{equation*}
    where $\lambda + 1 = (\lambda_1 + 1, \lambda _2, \ldots)$ for any partition $\lambda$.
\end{corollary}

\begin{proof}
    The characterization of $\pi$ for which $\ls_{\{1\}}(\pi) < \is(\pi) + 1$ can be derived using \Cref{gsievac}. 
    These are precisely the permutations where any longest increasing subsequence must use the first term of the permutation. 
    To count these permutations, we note that their reverse complements correspond, under RSK, to pairs $(P, Q)$ of SYTs of the same shape where the last entry of the first row of $Q$ is $n$.
\end{proof}

% \begin{proposition}
%     \textcolor{red}{formula for number of $\pi$ such that $\ls_{\{i\}}(\pi) = k$ using Littlewood Richardson}
% \end{proposition}

\subsection{The general case}

We now move on to the general statistic $\ls_D$. 
The proof of the following result is similar to that of \Cref{gsinonzero}.

\begin{proposition}
    Let $\pi$ be a non-identity permutation of size $n$. 
    Let $D = D' \sqcup [i, i + l - 1]$ where $i$ is the largest element of $D$ such that $i - 1 \notin D$. 
    Let $k$ be the largest index such that $\pi_{[k, n]}$ has a decreasing subsequence of length $l$. 
    We have, $\ls_D(\pi) \neq 0$ if and only if $i \leq \ls_{D'}(\pi_{[k]})$.
\end{proposition}

% \textcolor{red}{K: We can also omit the Knuth transformation stuff and just use the above proposition (and induction) to show that $\ls_D$ only depends on $Q$}

% \begin{proposition}
%     For any permutation $\pi \in \s_n$ and $m \geq 1$, we have $\ls_{[m]}(\pi) \neq 0$ if and only if $\is(\pi^r) \geq m + 1$. 
%     In this case, we have
%     \begin{equation*}
%         \ls_{[m]}(\pi) = m + \is(\pi_k \pi_{k + 1} \cdots \pi_n)
%     \end{equation*}
%     where $k$ is the smallest index such that $\is(\pi_k \pi_{k - 1} \cdots \pi_1) = m + 1$.
% \end{proposition}

The above proposition allows us to inductively determine if $\ls_D(\pi) \neq 0$ by breaking up $D$ into intervals. 
We now describe a method to extract $\ls_D(\pi)$ using this break up of $D$. 
We use the notation $\ds(\pi)$ to denote the length of a longest decreasing subsequence of $\pi$.

\begin{proposition}
    Let $\pi \in \s_n$ and $D \subseteq [n - 1]$ such that $\ls_D(\pi) \neq 0$. 
    If $D = \varnothing$, we have $\ls_D(\pi) = \is(\pi)$. 
    Otherwise, we have
    \begin{equation*}
        \ls_D(\pi) = m + \ls_{D'}(\pi_{[k, n]})
    \end{equation*}
    where $D' = \{i - m \mid i \in D \setminus [m]\}$ and
    \begin{itemize}
        \item if $1 \in D$, then $m$ is the smallest number such that $m + 1 \notin D$ and $k$ is the smallest index such that $\ds(\pi_{[k]}) = m + 1$, and
        \item if $1 \notin D$, then $m = \min D - 1$ and $k$ is the smallest index such that $\is(\pi_{[k]}) = m + 1$.
    \end{itemize}
\end{proposition}

\begin{proof}
    We only consider the case when $1 \in D$, the other case can be proved similarly. 
    By the definition of $m$ and $k$, we must have that $\ls_D(\pi) \leq m + \ls_{D'}(\pi_{[k, n]})$. 
    We prove the equality by constructing a subsequence of $\pi$ with descent set $D$ of the required length.

    Let $\pi_K$ be a decreasing subsequence of length $m + 1$ with $\max K = k$ and $\pi_J$ be a subsequence with $J \subseteq [k, n]$ and $\Des(\pi_J) = D'$. 
    Let $j = \min J$. 
    If $\pi_k \leq \pi_j$, then consider $I = (K \setminus \{k\}) \cup J$ and note that $\Des(\pi_I) = D$. 
    Similarly, if $\pi_k > \pi_j$, then consider $I = K \cup (J \setminus \{j\})$.
\end{proof}

We unravel the inductive method mentioned above to describe a procedure to extract $\ls_D(\pi)$ directly. 
To do this, it is convenient to represent the set $D$ as a composition. 
To any composition $c = (c_1, c_2, \ldots, c_k)$, we associate a set $D_c$ as follows: 
Set $d_i = c_1 + c_2 + \cdots + c_i$ for all $i \in [k]$. 
If $k$ is odd, then $$D_c = [d_1] \cup [d_2 + 1, d_3] \cup \cdots \cup [d_{k - 1} + 1, d_k].$$
Similarly, if $k$ is even, then $$D_c = [d_1 + 1, d_2] \cup [d_3 + 1, d_4] \cup \cdots \cup [d_{k - 1} + 1, d_k].$$
For example, $D_{(2, 3, 1)} = \{1, 2, 6\}$.

To extract $\ls_D(\pi)$, we consider a triangle of pairs that record the longest increasing and decreasing subsequence lengths for factors of $\pi$.
\begin{definition}
    For a permutation $\pi \in \s_n$ and $1 \leq i \leq j \leq n$,
    \begin{equation*}
        a_{i, j}(\pi) \coloneqq (\is(\pi_{[i, j]}), \ds(\pi_{[i, j]})).
    \end{equation*}
\end{definition}

We draw the triangle of $a_{i, j}(\pi)$ as shown in \Cref{lsaijex1}. 
We use blue entries for longest increasing subsequence lengths and red for decreasing. 
If the triangle is rotated $45$ degrees clockwise, then the labeling $i, j$ is just like that of a matrix. 
Note that the values $a_{i, j}(\pi)$ are the lengths of the firsts rows and columns of the (nonempty) partitions in the evacuation growth diagram of $Q(\pi)$.

\begin{example}
    For $\pi = 314526$, we have $a_{2, 5}(\pi) = ({\color{blue}3}, {\color{red}2})$ since the length of a longest increasing subsequence of $1452$ is $3$ and that of a longest decreasing subsequence is $2$.
\end{example}

We now describe an algorithm to extract $\ls_D(\pi)$ using the triangle of $a_{i, j}(\pi)$ values.

\begin{algorithm}\label{lsDalgo}
Given a composition $c = (c_1, \ldots, c_k)$ and a permutation $\pi$, we extract $\ls_{D_c}(\pi)$ as follows. 
We consider the case when $k$ is even, the other case is analogous.
\begin{enumerate}
    \item Find the smallest value of $j_1$ such that the second entry of $a_{1, j_1}(\pi)$ is $c_1 + 1$. 
    If no such $j_1$ exists, then $\ls_{D_c}(\pi) = 0$.

    \item If $k > 1$, find the smallest value of $j_2$ such that the first entry of $a_{j_1, j_2}(\pi)$ is $c_2 + 1$. 
    If no such $j_2$ exists, then $\ls_{D_c}(\pi) = 0$.

    \item Continue this way till either we get $\ls_{D_c}(\pi) = 0$ at some step or we find $j_k$. 
    If the first entry of $a_{j_k, n}(\pi)$ is $l$, then $\ls_{D_c}(\pi) = c_1 + c_2 + \cdots + c_k + l$.
\end{enumerate}
\end{algorithm}

\begin{example}
    Let $\pi = 3247516$ and $D = \{2, 3\}$. 
    To extract $\ls_D(\pi)$ from the $a_{i, j}(\pi)$ triangle, note that $D = D_{(1, 2)}$ (see \Cref{lsaijex1}). 
    The steps to extract $\ls_D(\pi)$ is as follows.
    \begin{enumerate}
        \item We move up from the bottom-left entry until we see a term with first (blue) entry $2 = c_1 + 1$. 
        Here we find it at $a_{1, 3}(\pi)$.

        \item  We then move down to the bottom of the triangle. 
        We then move up until we see a term with second (red) entry $3 = c_2 + 1$. 
        Here we find it at $a_{3, 6}(\pi)$.
        
        \item We then move down to the bottom of the triangle. 
        Now we keep moving up for as much as we can, and then record the first (blue) entry in this term. 
        Here it is $a_{6, 7}(\pi)$, whose first entry is $2$. 
        Hence, we have $\ls_D(\pi) = 1 + 2 + 2 = 5$.
    \end{enumerate}
\end{example}

\begin{figure}[ht]
    \centering
    \begin{tikzpicture}[scale = 1.2, yscale = 1.4, rotate = -45]
        \node (a11) at (1, 1) {\scriptsize$( {\color{blue}1}, {\color{red}1} )$};
        \node (a12) at (1, 2) {\scriptsize$( {\color{blue}1}, {\color{red}2} )$};
        \node (a13) at (1, 3) {\scriptsize$( {\color{blue}2}, {\color{red}2} )$};
        \node (a14) at (1, 4) {\scriptsize$( {\color{blue}3}, {\color{red}2} )$};
        \node (a15) at (1, 5) {\scriptsize$( {\color{blue}3}, {\color{red}2} )$};
        \node (a16) at (1, 6) {\scriptsize$( {\color{blue}3}, {\color{red}3} )$};
        \node (a17) at (1, 7) {\scriptsize$( {\color{blue}4}, {\color{red}3} )$};
        \node (a22) at (2, 2) {\scriptsize$( {\color{blue}1}, {\color{red}1} )$};
        \node (a23) at (2, 3) {\scriptsize$( {\color{blue}2}, {\color{red}1} )$};
        \node (a24) at (2, 4) {\scriptsize$( {\color{blue}3}, {\color{red}1} )$};
        \node (a25) at (2, 5) {\scriptsize$( {\color{blue}3}, {\color{red}2} )$};
        \node (a26) at (2, 6) {\scriptsize$( {\color{blue}3}, {\color{red}3} )$};
        \node (a27) at (2, 7) {\scriptsize$( {\color{blue}4}, {\color{red}3} )$};
        \node (a33) at (3, 3) {\scriptsize$( {\color{blue}1}, {\color{red}1} )$};
        \node (a34) at (3, 4) {\scriptsize$( {\color{blue}2}, {\color{red}1} )$};
        \node (a35) at (3, 5) {\scriptsize$( {\color{blue}2}, {\color{red}2} )$};
        \node (a36) at (3, 6) {\scriptsize$( {\color{blue}2}, {\color{red}3} )$};
        \node (a37) at (3, 7) {\scriptsize$( {\color{blue}3}, {\color{red}3} )$};
        \node (a44) at (4, 4) {\scriptsize$( {\color{blue}1}, {\color{red}1} )$};
        \node (a45) at (4, 5) {\scriptsize$( {\color{blue}1}, {\color{red}2} )$};
        \node (a46) at (4, 6) {\scriptsize$( {\color{blue}1}, {\color{red}3} )$};
        \node (a47) at (4, 7) {\scriptsize$( {\color{blue}2}, {\color{red}3} )$};
        \node (a55) at (5, 5) {\scriptsize$( {\color{blue}1}, {\color{red}1} )$};
        \node (a56) at (5, 6) {\scriptsize$( {\color{blue}1}, {\color{red}2} )$};
        \node (a57) at (5, 7) {\scriptsize$( {\color{blue}2}, {\color{red}2} )$};
        \node (a66) at (6, 6) {\scriptsize$( {\color{blue}1}, {\color{red}1} )$};
        \node (a67) at (6, 7) {\scriptsize$( {\color{blue}2}, {\color{red}1} )$};
        \node (a77) at (7, 7) {\scriptsize$( {\color{blue}1}, {\color{red}1} )$};
        \draw[ultra thick, blue] (a11) -- (a12);
        \draw[ultra thick, blue] (a12) -- (a13);
        \draw (a13) -- (a14);
        \draw (a14) -- (a15);
        \draw (a15) -- (a16);
        \draw (a16) -- (a17);
        \draw (a22) -- (a23);
        \draw (a23) -- (a24);
        \draw (a24) -- (a25);
        \draw (a25) -- (a26);
        \draw (a26) -- (a27);
        \draw[ultra thick, red] (a33) -- (a34);
        \draw[ultra thick, red] (a34) -- (a35);
        \draw[ultra thick, red] (a35) -- (a36);
        \draw (a36) -- (a37);
        \draw (a44) -- (a45);
        \draw (a45) -- (a46);
        \draw (a46) -- (a47);
        \draw (a55) -- (a56);
        \draw (a56) -- (a57);
        \draw[ultra thick, blue] (a66) -- (a67);
        \draw (a12) -- (a22);
        \draw[ultra thick, blue] (a13) -- (a23);
        \draw[ultra thick, blue] (a23) -- (a33);
        \draw (a14) -- (a24);
        \draw (a24) -- (a34);
        \draw (a34) -- (a44);
        \draw (a15) -- (a25);
        \draw (a25) -- (a35);
        \draw (a35) -- (a45);
        \draw (a45) -- (a55);
        \draw (a16) -- (a26);
        \draw (a26) -- (a36);
        \draw[ultra thick, red] (a36) -- (a46);
        \draw[ultra thick, red] (a46) -- (a56);
        \draw[ultra thick, red] (a56) -- (a66);
        \draw (a17) -- (a27);
        \draw (a27) -- (a37);
        \draw (a37) -- (a47);
        \draw (a47) -- (a57);
        \draw (a57) -- (a67);
        \draw[ultra thick, blue] (a67) -- (a77);
        % \node at (1.25, 0.75) {\tiny 1};
        % \node at (2.25, 1.75) {\tiny 2};
        % \node at (3.25, 2.75) {\tiny 3};
        % \node at (4.25, 3.75) {\tiny 4};
        % \node at (5.25, 4.75) {\tiny 5};
        % \node at (6.25, 5.75) {\tiny 6};
        % \node at (7.25, 6.75) {\tiny 7};
    \end{tikzpicture}
    \caption{Extracting $\ls_{\{2, 3\}}(\pi) = 5$ from the $a_{i, j}(\pi)$ triangle.}
    \label{lsaijex1}
\end{figure}
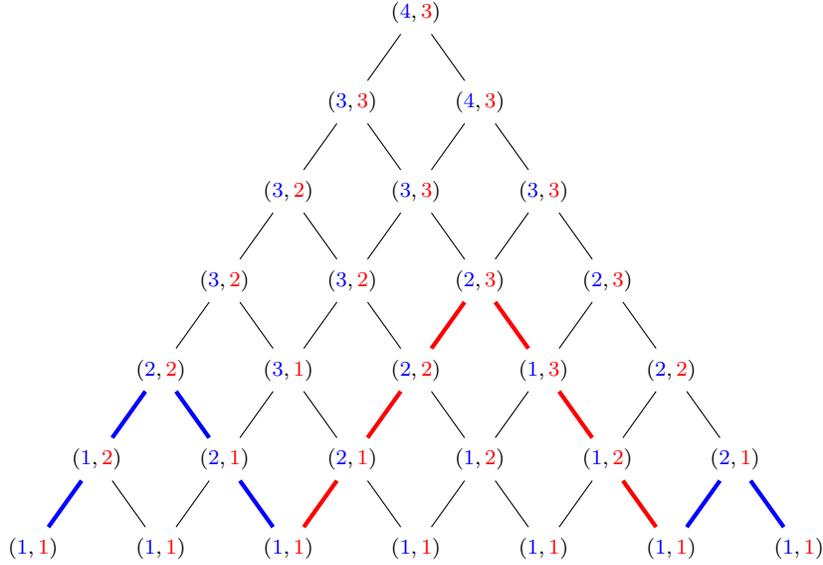

\begin{example}
    Suppose $\pi$ is a permutation such that the $a_{i, j}(\pi)$ triangle is given by \Cref{lsaijex2}. 
    Then we must have $\ls_{\{1, 2, 4, 5\}}(\pi) = 0$ since we cannot finish the procedure to extract $\ls_D(\pi)$.
\end{example}

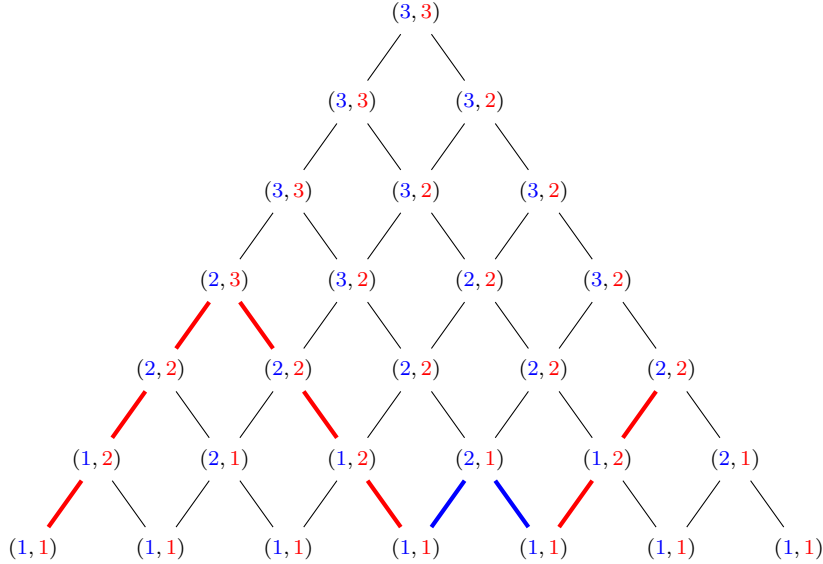
\begin{figure}[H]
    \centering
    \begin{tikzpicture}[scale = 1.2, yscale = 1.4, rotate = -45]
        \node (a11) at (1, 1) {\scriptsize$( {\color{blue}1}, {\color{red}1} )$};
        \node (a12) at (1, 2) {\scriptsize$( {\color{blue}1}, {\color{red}2} )$};
        \node (a13) at (1, 3) {\scriptsize$( {\color{blue}2}, {\color{red}2} )$};
        \node (a14) at (1, 4) {\scriptsize$( {\color{blue}2}, {\color{red}3} )$};
        \node (a15) at (1, 5) {\scriptsize$( {\color{blue}3}, {\color{red}3} )$};
        \node (a16) at (1, 6) {\scriptsize$( {\color{blue}3}, {\color{red}3} )$};
        \node (a17) at (1, 7) {\scriptsize$( {\color{blue}3}, {\color{red}3} )$};
        \node (a22) at (2, 2) {\scriptsize$( {\color{blue}1}, {\color{red}1} )$};
        \node (a23) at (2, 3) {\scriptsize$( {\color{blue}2}, {\color{red}1} )$};
        \node (a24) at (2, 4) {\scriptsize$( {\color{blue}2}, {\color{red}2} )$};
        \node (a25) at (2, 5) {\scriptsize$( {\color{blue}3}, {\color{red}2} )$};
        \node (a26) at (2, 6) {\scriptsize$( {\color{blue}3}, {\color{red}2} )$};
        \node (a27) at (2, 7) {\scriptsize$( {\color{blue}3}, {\color{red}2} )$};
        \node (a33) at (3, 3) {\scriptsize$( {\color{blue}1}, {\color{red}1} )$};
        \node (a34) at (3, 4) {\scriptsize$( {\color{blue}1}, {\color{red}2} )$};
        \node (a35) at (3, 5) {\scriptsize$( {\color{blue}2}, {\color{red}2} )$};
        \node (a36) at (3, 6) {\scriptsize$( {\color{blue}2}, {\color{red}2} )$};
        \node (a37) at (3, 7) {\scriptsize$( {\color{blue}3}, {\color{red}2} )$};
        \node (a44) at (4, 4) {\scriptsize$( {\color{blue}1}, {\color{red}1} )$};
        \node (a45) at (4, 5) {\scriptsize$( {\color{blue}2}, {\color{red}1} )$};
        \node (a46) at (4, 6) {\scriptsize$( {\color{blue}2}, {\color{red}2} )$};
        \node (a47) at (4, 7) {\scriptsize$( {\color{blue}3}, {\color{red}2} )$};
        \node (a55) at (5, 5) {\scriptsize$( {\color{blue}1}, {\color{red}1} )$};
        \node (a56) at (5, 6) {\scriptsize$( {\color{blue}1}, {\color{red}2} )$};
        \node (a57) at (5, 7) {\scriptsize$( {\color{blue}2}, {\color{red}2} )$};
        \node (a66) at (6, 6) {\scriptsize$( {\color{blue}1}, {\color{red}1} )$};
        \node (a67) at (6, 7) {\scriptsize$( {\color{blue}2}, {\color{red}1} )$};
        \node (a77) at (7, 7) {\scriptsize$( {\color{blue}1}, {\color{red}1} )$};
        \draw[ultra thick, red] (a11) -- (a12);
        \draw[ultra thick, red] (a12) -- (a13);
        \draw[ultra thick, red] (a13) -- (a14);
        \draw (a14) -- (a15);
        \draw (a15) -- (a16);
        \draw (a16) -- (a17);
        \draw (a22) -- (a23);
        \draw (a23) -- (a24);
        \draw (a24) -- (a25);
        \draw (a25) -- (a26);
        \draw (a26) -- (a27);
        \draw (a33) -- (a34);
        \draw (a34) -- (a35);
        \draw (a35) -- (a36);
        \draw (a36) -- (a37);
        \draw[ultra thick, blue] (a44) -- (a45);
        \draw (a45) -- (a46);
        \draw (a46) -- (a47);
        \draw[ultra thick, red] (a55) -- (a56);
        \draw[ultra thick, red] (a56) -- (a57);
        \draw (a66) -- (a67);
        \draw (a12) -- (a22);
        \draw (a13) -- (a23);
        \draw (a23) -- (a33);
        \draw[ultra thick, red] (a14) -- (a24);
        \draw[ultra thick, red] (a24) -- (a34);
        \draw[ultra thick, red] (a34) -- (a44);
        \draw (a15) -- (a25);
        \draw (a25) -- (a35);
        \draw (a35) -- (a45);
        \draw[ultra thick, blue] (a45) -- (a55);
        \draw (a16) -- (a26);
        \draw (a26) -- (a36);
        \draw (a36) -- (a46);
        \draw (a46) -- (a56);
        \draw (a56) -- (a66);
        \draw (a17) -- (a27);
        \draw (a27) -- (a37);
        \draw (a37) -- (a47);
        \draw (a47) -- (a57);
        \draw (a57) -- (a67);
        \draw (a67) -- (a77);
        % \node at (1.25, 0.75) {\tiny 1};
        % \node at (2.25, 1.75) {\tiny 2};
        % \node at (3.25, 2.75) {\tiny 3};
        % \node at (4.25, 3.75) {\tiny 4};
        % \node at (5.25, 4.75) {\tiny 5};
        % \node at (6.25, 5.75) {\tiny 6};
        % \node at (7.25, 6.75) {\tiny 7};
    \end{tikzpicture}
    \caption{Showing $\ls_{\{1, 2, 4, 5\}}(\pi) = 0$ using the $a_{i, j}(\pi)$ triangle.}
    \label{lsaijex2}
\end{figure}

From the algorithm described above, we can see that for any $D$, the value $\ls_D(\pi)$ only depends on the values $a_{i, j}(\pi)$ as $i, j$ varies over $1 \leq i < j \leq n$. 
We have mentioned that these values can be extracted from the evacuation growth diagram of $Q(\pi)$. 
This gives us the following.

\begin{corollary}
    For any permutation $\pi \in \s_n$ and $D \subseteq [n - 1]$, the value $\ls_D(\pi)$ only depends on $Q(\pi)$.
\end{corollary}

However, the converse of the above result is not true. One can check that $\pi = 2147635$ and $\sigma = 3247615$ satisfy $\ls_D(\pi) = \ls_D(\sigma)$ for all sets $D$, but $Q(\pi) \neq Q(\sigma)$.

Although we can not extract $Q(\pi)$ from the values of $\ls_D(\pi)$, we can extract the triangle of values $a_{i, j}(\pi)$, which gives us the following result.

\begin{proposition}\label{equivalence}
    For any two permutations $\pi, \sigma \in \s_n$, the following are equivalent.
    \begin{itemize}
        \item For all $D \subseteq [n - 1]$, we have $\ls_D(\pi) = \ls_D(\sigma)$.

        \item For all $1 \leq i < j \leq n$, we have $a_{i, j}(\pi) = a_{i, j}(\sigma)$.
    \end{itemize}
\end{proposition}

\begin{proof}
    We only have to prove that we can obtain the values $a_{i, j}(\pi)$ from those of $\ls_D(\pi)$. 
    For a given $1 \leq i < j \leq n$, we show how one can obtain $\is(\pi_{[i, j]})$ using the values of $\ls_D(\pi)$. 
    The longest decreasing subsequence length in this factor can be found similarly.
    
    To do this, it is more convenient to use the notation of descent words (see \Cref{desword}). 
    Recall that we can use the values of $\ls_D(\pi)$ to determine if $\pi$ has a subsequence with a given descent word. 
    Let $w = \dword(\pi_{[i - 1]})$ and $v = \dword(\pi_{[j + 1, n]})$. 
    We find the largest value of $k$ such that $\pi$ has a subsequence whose descent word is of the form
    \begin{equation*}
        w\quad x_1 \quad \desword{U}^{k - 1}\quad x_2\quad v
    \end{equation*}
    where $x_1, x_2 \in \{\desword{U}, \desword{D}\}$ and $x_1$ is omitted if $i = 1$ and $x_2$ is omitted if $j = n$. 
    It can be verified that this value of $k$ is precisely $\is(\pi_{[i, j]})$.
\end{proof}
% \textcolor{red}{Checked for $n \leq 10$}

% \begin{proof}
%     We only have to show that if there exist $1 \leq i < j \leq n$ such that $a_{i, j}(\pi) \neq a_{i, j}(\sigma)$, then there exists $D \subseteq [n - 1]$ such that $\ls_D(\pi) \neq \ls_D(\sigma)$. 
%     Let $(i,j)$ be the lexicographically smallest pair such that $a_{i, j}(\pi) \neq a_{i, j}(\sigma)$.

%     We deal with the case when $i > 1$ and $k = \is(\pi_i\pi_{i + 1} \cdots \pi_j) < \is(\sigma_i\sigma_{i + 1} \cdots \sigma_j) = l$. 
%     The other cases can be dealt with similarly. 
%     Let $D' = \Des (\pi_1\pi_2 \cdots \pi_i)$ and by the definition of $i$, we also have $D' = \Des(\sigma_1\sigma_2 \cdots \sigma_i)$.    
% \end{proof}

\section{Future directions}\label{sec:fd}

We now gather some problems that could be of interest to explore.

\begin{enumerate}
    \item An obvious question would be to study the distributions of the statistics mentioned in this paper. 
    For example, are there reasonable ways to count permutations $\pi \in \s_n$ with $\ls_d(\pi) = k$ for given values of $n, d, k$? 
    It might be easier to solve this problem for certain sub-classes of permutations. 
    Restricting to the class of permutations with at most $1$ descent is done in \cite{grass}.

    \item The next direction is combining these statistics with pattern avoidance (basic definitions can be found in \cite[Chapter 4]{bona_book}). 
    Several questions involving the distribution of the statistics $\is$ and $\as$ over classes of pattern-avoiding permutations have been explored \cite{isav,asav,isasav}. 
    One could try to study the distribution of the statistics mentioned in this paper over pattern-avoiding permutations.
    
    Another question, mentioned in \cite[Section 5.3]{stan_as}, involves sets of patterns of the form $P_{k, D} \coloneqq \{\sigma \in \s_k \mid \Des(\sigma) = D\}$. 
    Note that a permutation $\pi$ avoids all the patterns in $P_{k, D}$ if and only if $\ls_D(\pi) < k$. 
    Since $\ls_D(\pi)$ only depends on $Q(\pi)$, it might be interesting to explore avoidance of these sets of patterns for involutions.

    \item \Cref{equivalence} gives two descriptions of an equivalence on permutations. 
    What more can be said about this equivalence? 
    How many equivalences classes are there among permutations of size $n$? 
    Are there any general methods to obtain non-trivial equivalences among permutations?
\end{enumerate}

% \textcolor{red}{If $P_{n, D} = \{\sigma \in \s_n \mid \Des(\sigma) = D\}$, then $\pi \in \operatorname{Av}(P_{n, D})$ if and only if $\ls_D(\pi) < n$ and in particular only depends on $Q(\pi)$. 
% The Stanley-Wilf limit of such patterns is asked in question 3 of Stanley's alternating subsequence paper}

% {\color{red}
% The distribution of the statistic on SYTs $T \rightarrow gs(\operatorname{RSK}^{-1}(T, T))$ for $n=12$ is not log-concave.}

\section*{Acknowledgements}

The first author is supported by the Göran Gustafsson Foundation and the Verg Foundation.  The second author was supported by the Research Initiation Grant (RIG) from IIT Bhilai.

\bibliographystyle{abbrv}
\bibliography{lsdrefs}

@book {stanley_ec2,
    AUTHOR = {Stanley, Richard P.},
     TITLE = {Enumerative combinatorics. {V}ol. 2},
   EDITION = {Second edition},
 PUBLISHER = {Cambridge University Press},
      YEAR = {2024},
     PAGES = {xvi+783},
}

@article {fomin_growth,
    AUTHOR = {Fomin, Sergey},
     TITLE = {Duality of graded graphs},
   JOURNAL = {J. Algebraic Combin.},
  FJOURNAL = {Journal of Algebraic Combinatorics. An International Journal},
    VOLUME = {3},
      YEAR = {1994},
    NUMBER = {4},
     PAGES = {357--404},
      ISSN = {0925-9899,1572-9192},
   MRCLASS = {05C50 (05C20 05C38 05E10)},
  MRNUMBER = {1293822},
MRREVIEWER = {Daniel\ Ashlock},
       DOI = {10.1023/A:1022412010826},
       URL = {https://doi.org/10.1023/A:1022412010826},
}

@article {schensted,
    AUTHOR = {Schensted, C.},
     TITLE = {Longest increasing and decreasing subsequences},
   JOURNAL = {Canadian J. Math.},
  FJOURNAL = {Canadian Journal of Mathematics. Journal Canadien de
              Math\'ematiques},
    VOLUME = {13},
      YEAR = {1961},
     PAGES = {179--191},
      ISSN = {0008-414X,1496-4279},
   MRCLASS = {05.00},
  MRNUMBER = {121305},
MRREVIEWER = {D.\ E.\ Rutherford},
       DOI = {10.4153/CJM-1961-015-3},
       URL = {https://doi.org/10.4153/CJM-1961-015-3},
}

@article {schutzenberger,
    AUTHOR = {Sch\"utzenberger, M. P.},
     TITLE = {Promotion des morphismes d'ensembles ordonn\'es},
   JOURNAL = {Discrete Math.},
  FJOURNAL = {Discrete Mathematics},
    VOLUME = {2},
      YEAR = {1972},
     PAGES = {73--94},
      ISSN = {0012-365X,1872-681X},
   MRCLASS = {06A15},
  MRNUMBER = {299539},
MRREVIEWER = {E.\ Harzheim},
       DOI = {10.1016/0012-365X(72)90062-3},
       URL = {https://doi.org/10.1016/0012-365X(72)90062-3},
}

@misc{oeis,
	author = {{OEIS Foundation Inc.}},
	note = {Published electronically at \url{http://oeis.org}},
	title = {The {O}n-{L}ine {E}ncyclopedia of {I}nteger {S}equences},
	year = 2025
}

@incollection {stan_as,
    AUTHOR = {Stanley, Richard P.},
     TITLE = {Longest alternating subsequences of permutations},
      NOTE = {Special volume in honor of Melvin Hochster},
   JOURNAL = {Michigan Math. J.},
  FJOURNAL = {Michigan Mathematical Journal},
    VOLUME = {57},
      YEAR = {2008},
     PAGES = {675--687},
      ISSN = {0026-2285,1945-2365},
   MRCLASS = {05A05},
  MRNUMBER = {2492475},
MRREVIEWER = {Mikl\'os\ B\'ona},
       DOI = {10.1307/mmj/1220879431},
       URL = {https://doi.org/10.1307/mmj/1220879431},
}

@article {grass,
    AUTHOR = {Menon, Krishna and Singh, Anurag},
     TITLE = {Dyck paths, binary words, and {G}rassmannian permutations
              avoiding an increasing pattern},
   JOURNAL = {Ann. Comb.},
  FJOURNAL = {Annals of Combinatorics},
    VOLUME = {28},
      YEAR = {2024},
    NUMBER = {3},
     PAGES = {871--887},
      ISSN = {0218-0006,0219-3094},
   MRCLASS = {05A05 (05A15 05A19)},
  MRNUMBER = {4787909},
MRREVIEWER = {Nik\ Lygeros},
       DOI = {10.1007/s00026-023-00667-x},
       URL = {https://doi.org/10.1007/s00026-023-00667-x},
}

@book {bona_book,
    AUTHOR = {B\'ona, Mikl\'os},
     TITLE = {Combinatorics of permutations},
    SERIES = {Discrete Mathematics and its Applications (Boca Raton)},
      NOTE = {With a foreword by Richard Stanley},
 PUBLISHER = {Chapman \& Hall/CRC, Boca Raton, FL},
      YEAR = {2004},
     PAGES = {xiv+383},
      ISBN = {1-58488-434-7},
   MRCLASS = {05A05 (05-01)},
  MRNUMBER = {2078910},
MRREVIEWER = {Martin\ Klazar},
       DOI = {10.1201/9780203494370},
       URL = {https://doi.org/10.1201/9780203494370},
}

@article {isav,
    AUTHOR = {Deutsch, Emeric and Hildebrand, A. J. and Wilf, Herbert S.},
     TITLE = {Longest increasing subsequences in pattern-restricted
              permutations},
      NOTE = {Permutation patterns (Otago, 2003)},
   JOURNAL = {Electron. J. Combin.},
  FJOURNAL = {Electronic Journal of Combinatorics},
    VOLUME = {9},
      YEAR = {2002/03},
    NUMBER = {2},
     PAGES = {Research paper 12, 8},
      ISSN = {1077-8926},
   MRCLASS = {05A16 (05A05)},
  MRNUMBER = {2028291},
MRREVIEWER = {Kendra\ Killpatrick},
       DOI = {10.37236/1684},
       URL = {https://doi.org/10.37236/1684},
}

@article {asav,
    AUTHOR = {Firro, Ghassan and Mansour, Toufik and Wilson, Mark C.},
     TITLE = {Longest alternating subsequences in pattern-restricted
              permutations},
   JOURNAL = {Electron. J. Combin.},
  FJOURNAL = {Electronic Journal of Combinatorics},
    VOLUME = {14},
      YEAR = {2007},
    NUMBER = {1},
     PAGES = {Research Paper 34, 17},
      ISSN = {1077-8926},
   MRCLASS = {05A05 (05A15 05A16 60C05)},
  MRNUMBER = {2302541},
MRREVIEWER = {Eric\ S.\ Egge},
       DOI = {10.37236/952},
       URL = {https://doi.org/10.37236/952},
}

@article {isasav,
    AUTHOR = {Madras, Neal and Yıldırım, G\"okhan},
     TITLE = {Longest monotone subsequences and rare regions of
              pattern-avoiding permutations},
   JOURNAL = {Electron. J. Combin.},
  FJOURNAL = {Electronic Journal of Combinatorics},
    VOLUME = {24},
      YEAR = {2017},
    NUMBER = {4},
     PAGES = {Paper No. 4.13, 29},
      ISSN = {1077-8926},
   MRCLASS = {05A05 (05A16 60F99)},
  MRNUMBER = {3711046},
MRREVIEWER = {David\ Bevan},
       DOI = {10.37236/6402},
       URL = {https://doi.org/10.37236/6402},
}

@book {lisbook,
    AUTHOR = {Romik, Dan},
     TITLE = {The surprising mathematics of longest increasing subsequences},
    SERIES = {Institute of Mathematical Statistics Textbooks},
    VOLUME = {4},
 PUBLISHER = {Cambridge University Press, New York},
      YEAR = {2015},
     PAGES = {xi+353},
      ISBN = {978-1-107-42882-9; 978-1-107-07583-2},
   MRCLASS = {05-01 (05A05 05D40 60B20 60C05 60K35 82B41 82C41)},
  MRNUMBER = {3468738},
MRREVIEWER = {Sergi\ Elizalde},
}

@article {almostinc,
    AUTHOR = {Elmasry, Amr},
     TITLE = {The longest almost-increasing subsequence},
   JOURNAL = {Inform. Process. Lett.},
  FJOURNAL = {Information Processing Letters},
    VOLUME = {110},
      YEAR = {2010},
    NUMBER = {16},
     PAGES = {655--658},
      ISSN = {0020-0190,1872-6119},
   MRCLASS = {68W32},
  MRNUMBER = {2676799},
       DOI = {10.1016/j.ipl.2010.05.022},
       URL = {https://doi.org/10.1016/j.ipl.2010.05.022},
}

@article {common,
    AUTHOR = {Yang, I-Hsuan and Huang, Chien-Pin and Chao, Kun-Mao},
     TITLE = {A fast algorithm for computing a longest common increasing
              subsequence},
   JOURNAL = {Inform. Process. Lett.},
  FJOURNAL = {Information Processing Letters},
    VOLUME = {93},
      YEAR = {2005},
    NUMBER = {5},
     PAGES = {249--253},
      ISSN = {0020-0190,1872-6119},
   MRCLASS = {68W40 (68Q25 68W05 92D20)},
  MRNUMBER = {2112994},
       DOI = {10.1016/j.ipl.2004.10.014},
       URL = {https://doi.org/10.1016/j.ipl.2004.10.014},
}

@article {wave,
    AUTHOR = {Chen, Guan-Zhi and Yang, Chang-Biau and Chang, Yu-Cheng},
     TITLE = {The longest wave subsequence problem: generalizations of the
              longest increasing subsequence problem},
   JOURNAL = {Internat. J. Found. Comput. Sci.},
  FJOURNAL = {International Journal of Foundations of Computer Science},
    VOLUME = {36},
      YEAR = {2025},
    NUMBER = {2},
     PAGES = {203--218},
      ISSN = {0129-0541,1793-6373},
   MRCLASS = {68W32},
  MRNUMBER = {4871143},
MRREVIEWER = {Sergej\ V.\ Znamenskij},
       DOI = {10.1142/S012905412450014X},
       URL = {https://doi.org/10.1142/S012905412450014X},
}

@incollection {stanley2007increasing,
    AUTHOR = {Stanley, Richard P.},
     TITLE = {Increasing and decreasing subsequences and their variants},
 BOOKTITLE = {International {C}ongress of {M}athematicians. {V}ol. {I}},
     PAGES = {545--579},
 PUBLISHER = {Eur. Math. Soc., Z\"urich},
      YEAR = {2007},
      ISBN = {978-3-03719-022-7},
   MRCLASS = {05A05 (05C70 05E10 60C05)},
  MRNUMBER = {2334203},
MRREVIEWER = {Sylvie\ Corteel},
       DOI = {10.4171/022-1/21},
       URL = {https://doi.org/10.4171/022-1/21},
}

@article {greene,
    AUTHOR = {Greene, Curtis},
     TITLE = {An extension of {S}chensted's theorem},
   JOURNAL = {Advances in Math.},
  FJOURNAL = {Advances in Mathematics},
    VOLUME = {14},
      YEAR = {1974},
     PAGES = {254--265},
      ISSN = {0001-8708},
   MRCLASS = {05A17},
  MRNUMBER = {354395},
MRREVIEWER = {A.\ O.\ Morris},
       DOI = {10.1016/0001-8708(74)90031-0},
       URL = {https://doi.org/10.1016/0001-8708(74)90031-0},
}

\end{document}